\documentclass[11pt,reqno]{amsart}
\usepackage{hyperref}
\usepackage[top=3cm, bottom=3cm, left=3cm, right=3cm] {geometry}
\usepackage{amsmath, amssymb, latexsym, amscd, amsthm, amsfonts, amstext}
\usepackage[mathscr]{eucal}
\usepackage{xspace}
\usepackage{xcolor} 
\makeatletter
\@namedef{subjclassname@2020}{%
  \textup{2020} Mathematics Subject Classification}
\makeatother
\theoremstyle{plain}
\newtheorem{theorem}{Theorem}
\newtheorem{lemma}{Lemma}
\newtheorem{remark}{Remark}

\theoremstyle{definition}

\allowdisplaybreaks
\begin{document}

\title[Singular $p$-biharmonic problem with the Hardy potential]{
Singular $p$-biharmonic problem with the Hardy potential}

\centerline{}
\author{A. Drissi}
\address{A. Drissi, \newline
Department of Mathematics, Faculty of Sciences, Tunis El Manar University, Tunis 2092, Tunisia.}
\email{\href{mailto: A. Drissi <amor.drissi@ipeiem.utm.tn>}{amor.drissi@ipeiem.utm.tn}}
\author{A. Ghanmi}
\address{A. Ghanmi, \newline
Department of Mathematics, Faculty of Sciences, Tunis El Manar University, Tunis 2092, Tunisia.}
\email{\href{mailto: A. Ghanmi <abdeljabbar.ghanmi@gmail.com>}{abdeljabbar.ghanmi@gmail.com}}
\author{D.D. Repov\v{s}}
\address{D.D. Repov\v{s}, \newline
Faculty of Education, University of Ljubljana, 1000 Ljubljana, Slovenia.\newline
Faculty of Mathematics and Physics, University of Ljubljana, 1000 Ljubljana,
Slovenia.\newline
Institute of Mathematics, Physics and Mechanics, 1000 Ljubljana, Slovenia.}
\email{\href{mailto: D.D. Repov\v{s} <dusan.repovs@guest.arnes.si.>}{dusan.repovs@guest.arnes.si}}
\thanks{{\it Corresponding author:} Du\v{s}an D. Repov\v{s}}
\date{}
\thanks{The third author was supported by the Slovenian Research and Innovation Agency program P1-0292 and grants N1-0278, N1-0114, N1-0083, J1-4001, and J1-4031.}
\begin{abstract}
The aim of this paper is to study existence results for a singular problem involving the $p$-biharmonic operator and the Hardy potential. More precisely,  by combining monotonicity arguments with the variational method, the existence of solutions is established. By using the Nehari manifold method, the multiplicity of solutions is proved. An example is also given, to illustrate the importance of these results.
\end{abstract}
\maketitle

\section{Introduction}
The aim of this work  is to study  the  following    $p$-biharmonic problem with  singular nonlinearity and  Hardy potential
\begin{equation}\label{p}
\Delta_{p}^{2}\varphi-\lambda \frac{|\varphi|^{p-2}\varphi}{|z|^{2p}}+\Delta_p \varphi = \frac{a(z)}{ \varphi^{\theta}}+ \mu g(z,\varphi), \;\; \hbox{for all} \
\varphi\in W^{2,p}(\mathbb{R}^N),
\end{equation}
where $1<p<\frac{N}{2}, 0<\theta<1,$ and $\lambda$,  $\mu$ are positive constants. The operators $\Delta_{p}$ and $\Delta_{p}^{2}$ are  the $p$-Laplacian operator and the $p$-biharmonic operator, respectively, defined by
$$\Delta_{p}\varphi=\mbox{div}(|\nabla \varphi|^{p-2}\nabla \varphi)\;\;\mbox{and }\;\;\Delta^2_{p}\varphi=\Delta(|\Delta \varphi|^{p-2}\Delta \varphi).$$

Nonlinear elliptic equations with singularities can model several phenomena like non-Newtonian fluids, and chemical heterogeneity,  for more details and other applications, see for example, 
Alsaedi  et al. \cite{dhif1},
Callegari and Nachman \cite{ca1},  
Candito et al.  \cite{cand2, cand1}, 
Molica Bisci and R\v{a}dulescu \cite{mbr},  
Nachman and Callegari \cite{ca2}
Papageorgiou \cite{pap},  
Papageorgiou et al. \cite{pdc},
 and 
Pimenta and  Servadei \cite{mto}.
In recent years,  problems involving $p$-biharmonic operator have been extensively studied,  see for instance 
Bhakta \cite{bha}, 
Dhifli and Alsaedi \cite{dhif},  
Huang and Liu \cite{hwa}, 
Molica Bisci and Repov\v{s} \cite{mol-rep},  
Sun et al. \cite{sun1}, 
Wang and Zhao \cite{wwa}, 
and 
Yang et al.  \cite{rya}. 
In particular, 
Dhifli and Alsaedi  \cite{dhif} 
considered the analysis of the fibering map on the   Nehari manifold sets to  prove the existence of  multiple  solutions for the following  system
$$
\Delta_{p}^{2}\varphi-\Delta_{p}\varphi+V(z)  |\varphi|^{p-2}\varphi = \lambda f(z)  |\varphi|^{q-2}\varphi+a(z)  |\varphi|^{m-2}\varphi, \
\hbox{for all} \
\varphi\in W^{2,p}(\mathbb{R}^N).
$$

Very recently, several researchers
have concentrated on the study of singular  $p$-biharmonic equations, see
  Sun et al. \cite{sun1}
   and 
   Sun and Wu \cite{ sun2, sun}, 
   whereas singular problem involving $p$-biharmonic operator and Hardy potential has not received that much attention - we refer the reader to    
   Drissi et al. \cite{DGR}  
   and
   Huang and Liu \cite{hwa}   
    for related work.\\
Ferrara and Molica Bisci  \cite{ferr}
  used the variational principle of 
   Ricceri \cite{ricc}
    to prove the  multiplicity of solutions
 for  the following problem
$$\left\{
  \begin{array}{ll}
-\Delta_{p}\varphi  = \mu \frac{|\varphi|^{p-2}\varphi}{|z|^{2p}}+\lambda f(z,\varphi)
 \quad \mbox{ in } \Omega,\\
\varphi=\Delta \varphi=0 \quad
 \quad \quad \quad  \quad \quad \quad  \quad
\mbox{ on } \partial \Omega.
  \end{array}
\right.$$

Motivated by  \cite{ferr},
 Huang and Liu \cite{hwa}
  considered  the following $p$-biharmonic  problem
$$\left\{
  \begin{array}{ll}
-\Delta_{p}^2 \varphi - \mu \frac{|\varphi|^{p-2}\varphi}{|z|^{2p}}=\mu h(z,\varphi) \quad \mbox{ in } \Omega,\\
\varphi=\Delta \varphi=0 \quad
 \quad \quad  \quad \quad \quad  \quad \quad
 \mbox{ on } \partial \Omega.
  \end{array}
\right.$$
More precisely, they used the invariant sets of descending flows method and proved that under suitable conditions on the parameter $\mu$ and the nonlinearity $h$, such a problem admits a nontrivial solution that changes sign.

In the present paper, we shall combine variational methods with monotonicity arguments to prove the existence of a nontrivial solution for problem  \eqref{p}. Next, we shall use the Nehari manifold method to prove the multiplicity of solutions. We note that this problem   is very important since it involves the $p$-biharmonic operator, the $p$-Laplacian operator, a  singular nonlinearity, and the Hardy potential.\\
In the first main result of this paper, we shall
assume that
$$g(z,\varphi)=f(z)h(\varphi), \mbox{ for all }  (z, \varphi)\in\mathbb{R}^N \times \mathbb{R},$$
and  that the functions $f$, $h$ are measurable and satisfy the  following hypotheses:

$(H_1)$  There exist $c_1>0$, $1<r<p<\frac{N}{2}$  and $s \in (\frac{p^\ast}{p^\ast-r}, \frac{p}{p-r})$, such that
 $$f \in L^{\frac{p^\ast}{p^\ast-r}}(\mathbb{R}^N)\cap L^s_{loc}(\mathbb{R}^N)
\hbox{  and  }
   h(\varphi)\leq c_1 |\varphi|^{r-1}, \mbox{ for all } \varphi\in \mathbb{R}.$$

$(H_2)$ There exists  $M>0$   such that for all $ (z, \varphi)\in\mathbb{R}^N \times \mathbb{R}$,  we have
\begin{equation*}
0<r f(z)H(\varphi)\leq f(z)h(\varphi)\varphi,\text{ for all }\left\vert \varphi\right\vert \geq M,
\hbox{ where } H(t)=\int_{0}^{t}h(s)ds.
\end{equation*}

$(H_3)$ $a \in L^{\frac{p^\ast}{p^\ast+\theta-1}}(\mathbb{R}^N)\cap L^\beta_{loc}(\mathbb{R}^N)$, for some  $\beta \in (\frac{p^\ast}{p^\ast+\theta-1}, \frac{p}{\theta+p-1})$ .\\

The first main result of this paper is the following theorem.
\begin{theorem}\label{thm} Suppose that hypotheses $(H_1)$-$(H_3)$ hold.  Then
for all $\delta, \mu>0$, problem \eqref{p}  admits  at least one nontrivial weak solution $\varphi_\mu$, provided that $\lambda>0$ is small enough.
\end{theorem}

In the second main result of this paper, we shall assume the following hypotheses:\\

$(H_4)$  $ G:\mathbb{R}^N\times\mathbb{R}\longrightarrow\mathbb{R},$  defined by $G(z,\varphi)=\int_{0}^\varphi g(z,s)ds$, is a $C^1$ function such that $$G(z,t\varphi)=t^rG(z,\varphi), \mbox{ for all } (z,\varphi)\in \mathbb{R}^N\times\mathbb{R},  t>0.$$
Moreover, if $\varphi\neq 0$, then $G(z,\varphi)>0$,  where $0<1-\theta<1<p< r.$  \\

$(H_5)\; $  $a: \mathbb{R}^N\longrightarrow (0, \infty)$ satisfies
$$a \in L^{\frac{p}{\theta+p-1}}(\mathbb{R}^N).$$

We note that by hypothesis $(H_4)$,   we can find  $M>0$ such that
\begin{equation}
 \label{M}
 \varphi g(z,\varphi)=r G(z,\varphi)  \mbox{ and } \left|G(z,\varphi)\right|\le M|\varphi|^r,\;\;\mbox{for all }\;
 (z,\varphi)\in \mathbb{R}^N\times\mathbb{R}.
\end{equation}

The second main result of this paper is the following theorem.
\begin{theorem}
 \label{thmm}
Assume that  hypotheses $(H_4)$ and $(H_5)$ hold. Then there exists $\mu^*>0$ such that for all  $\mu\in (0,\mu^*)$, problem  \eqref{p} admits two nontrivial solutions.
\end{theorem}

The paper is organized as follows:
In Section \ref{s0} we shall present some preliminary material needed in the paper.
 In Section \ref{s2}  we shall prove the first main result of this paper, i.e. the existence of solutions (Theorem \ref{thm}). 
 In Section \ref{s00} we shall study fibering maps on Nehari manifold sets.
 In Section \ref{s3}, we shall prove the second main result of this paper, i.e. the multiplicity of solutions (Theorem \ref{thmm}). In Section \ref{s5} we shall give an illustrative example.

\section{Preliminaries}\label{s0}
In this section  we shall present some preliminary material needed in the paper. For other necessary background facts we recommend the comprehensive monograph 
Papageorgiou et al. \cite{PRR}.

The Hardy potential is related to  the following Rellich inequality
\begin{equation}\label{rellich}
  \int_{\mathbb{R}^N}\frac{|\varphi(z)|^p}{|z|^{2p}}\,dz\leq \left(\frac{p^2}{N(p-1)(N-2p)}\right)^p\int_{\mathbb{R}^N}|\Delta \varphi(z)|^p\,dz,\;\mbox{for all }\;\varphi\in E,
\end{equation}
where $E:=W^{2,p}(\mathbb{R}^N)$ is the Sobolev space which is defined as follows
$$W^{2,p}(\mathbb{R}^N) =\left\{\varphi\in L^p(\mathbb{R}^N): \Delta \varphi , \;|\nabla \varphi| \in L^p(\mathbb{R}^N)\right\}.$$
For the interested reader, properties  of
these spaces can be found in 
 Davies and Hinz \cite{davi}, 
 Mitidieri \cite{miti}, 
 and 
 Rellich \cite{reli}.
According to the Rellich inequality \eqref{rellich}, if $\lambda$ satisfies
\begin{equation} \label{lamb}0<\lambda < \left(\frac{N(p-1)(N-2p)}{p^2}\right)^p,
\end{equation}
then $\|.\| :E \to \mathbb{R},$ defined by
$$\|\varphi\|=\left(\int_{\mathbb{R}^N}|\Delta \varphi(z)|^p-\lambda \frac{|\varphi(z)|^p}{|z|^{2p}}+|\nabla \varphi(z)|^p\,dz\right)^{\frac{1}{p}},$$
is a norm in $E$.

For every
$r\in[p, p^\ast],$ there exists a continuous embedding from $E$ into  $L^{r}(\mathbb{R}^N)$. On the other hand, if  $r\in(p, p^\ast)$, then there exists a compact  embedding from $E$ into  $L^{r}_{loc}(\mathbb{R}^N)$. Moreover,  we have
\begin{equation}\label{const}
S_r |\varphi|^p_r \leq \|\varphi\|^p,
\ \mbox{for all }\varphi\in E \;\mbox{and }\;r\in[p, p^\ast],
\end{equation}
where $p^\ast=\frac{Np}{N-2p},$ $|\varphi|_r$ denotes the usual $L^{r}(\mathbb{R}^N)$-norm
 and   $S_r$ is the best Sobolev constant given by
$$S_r=\displaystyle\inf_{\varphi\in W^{2,p}(\mathbb{R}^N)\backslash\{0\}}\frac{\int_{\mathbb{R}^N}|\Delta \varphi(z)|^p-\lambda \frac{|\varphi(z)|^p}{|z|^{2p}}+|\nabla \varphi(z)|^p\,dz}{\left(\int_{\mathbb{R}^N}|\varphi(z)|^r\,dz\right)^{\frac{p}{r}}}.$$

If $\psi$ is a positive function on $\mathbb{R}^N$ and $1\leq \sigma <\infty$, then we can define the weighted Lebesgue space $L^\sigma(\mathbb{R^N}, \psi)$ by
$$L^\sigma(\mathbb{R^N}, \psi)=\left\{\varphi:\mathbb{R}^N \to \mathbb{R}\;\;\mbox{measurable }:\;
\int_{\mathbb{R}^N} \psi(z)|\varphi(z)|^\sigma\,dz<\infty\right\},$$
endowed with the norm
$$\|\varphi\|_{\sigma, \psi}=\left(\int_{\mathbb{R}^N} \psi(z)|\varphi(z)|^\sigma\,dz\right)^{\frac{1}{\sigma}}.$$
Then $L^\sigma(\mathbb{R^N}, \psi)$ is a uniformly convex Banach space.
Dhifli and Alsaedi \cite{dhif}
  have proved that
if $\psi  \in L^{\frac{p^\ast}{p^\ast-r}}(\mathbb{R}^N)\cap L^s_{loc}(\mathbb{R}^N)$, for some  $s \in (\frac{p^\ast}{p^\ast-r}, \frac{p}{p-r})$, then  the embedding $W^{2,p}(\mathbb{R}^N) \hookrightarrow L^r(\mathbb{R}^N, \psi)$ is continuous and compact. Moreover, we have the following estimate
 \begin{equation}\label{esti}
\|\varphi\|^r_{r, \psi}\leq  S_{p^*}^{-\frac{r}{p}}|f|_{\frac{p^\ast}{p^\ast-r}}\|\varphi\|^r,\;\;\mbox{for   all}\;\;\varphi\in  E.
\end{equation}
\begin{remark}\label{rem}
We get an  inequality similar to
  \eqref{esti}  if we replace $r$ by $1-\theta$ and $f$ by $a$. More precisely, we have
$$\int_{\mathbb{R}^N} a(z)|\varphi(z)|^{1-\theta}\,dz\leq  S_{p^*}^{-\frac{1-\theta}{p}}|f|_{\frac{p^\ast}{p^\ast+\theta-1}}\|\varphi\|^{1-\theta}.$$
Indeed, from equation \eqref{const} and the H\"older inequality, we obtain
\begin{eqnarray*}
\int_{\mathbb{R}^N} a(z)|\varphi(z)|^{1-\theta}\,dz&\leq&\left(\int_{\mathbb{R}^N} |a(z)|^{\frac{p^\ast}{p^\ast+\theta-1}}\right)^{\frac{p^\ast+\theta-1}{p^\ast}}  \left(\int_{\mathbb{R}^N} |u(z)|^{p^\ast}\right)^{\frac{1-\theta}{p^\ast}}\\
&\leq&  S_{p^*}^{-\frac{1-\theta}{p}}|f|_{\frac{p^\ast}{p^\ast+\theta-1}}\|\varphi\|^{1-\theta}.
\end{eqnarray*}
\end{remark}
\section{The proof of Theorem \ref{thm}}\label{s2}

We recall that a function  $\varphi\in E$ is called
 a weak solution for problem \eqref{p}, if   for all $v\in E$, one has
$$\int_{\mathbb{R}^N}|\Delta \varphi|^{p-2}\Delta \varphi \Delta v-\lambda \frac{|\varphi|^{p-2}\varphi v}{|z|^{2p}}+|\nabla \varphi|^{p-2}\nabla \varphi \nabla v\;dz$$
$$=\int_{\mathbb{R}^N}a(z)\varphi^{-\theta} v\,dz+\mu \int_{\mathbb{R}^N} g(z,\varphi) v\,dz.$$
Associated to problem \eqref{p}, we define the energy functional $J_\mu: E \to \mathbb{R}$ by
\begin{equation}\label{jmu}J_\mu(\varphi)=\frac{1}{p}\|\varphi\|^p-\frac{1}{1-\theta}\int_{\mathbb{R}^N}a(z) \varphi^{1-\theta}\,dz-\mu \int_{\mathbb{R}^N} G(z,\varphi(z))\,dz.\end{equation}
Several lemmas will  be needed for the proof of Theorem \ref{thm}.
\begin{lemma}\label{coer}Under hypotheses $(H_1)$-$(H_3)$, the  
 functional $J_{\mu}$ is coercive and bounded from below on $E$.
\end{lemma}
\begin{proof}Let $\varphi\in E$. Assume that the  hypotheses  $(H_1)$-$(H_3)$ hold. Then it follows by \eqref{esti}  and Remark \ref{rem} that
\begin{eqnarray*}
J_\mu(\varphi)&=&\frac{1}{p}\|\varphi\|^p-\frac{1}{1-\theta}\int_{\mathbb{R}^N}a(z) \varphi^{1-\theta}\,dz-\mu\int_{\mathbb{R}^N}f(z) H(\varphi)\,dz\\
&\geq& \frac{1}{p}\|\varphi\|^p-\frac{ S_{p^\ast}^{-\frac{1-\theta}{p}}}{1-\theta}|a|_{\frac{p^\ast}{p^\ast+\theta-1}}\|\varphi\|^{1-\theta}-\frac{\mu}{r}\|\varphi\|_{r, h}^r\\
&\geq&  \frac{1}{p}\|\varphi\|^p-\frac{ S_{p^\ast}^{-\frac{1-\theta}{p}}}{1-\theta}|a|_{\frac{p^\ast}{p^\ast+\theta-1}}\|\varphi\|^{1-\theta}-\frac{\mu  S_{p^\ast}^{-\frac{r}{p}}}{r}|f|_{\frac{p^\ast}{p^\ast-r}}\|\varphi\|^{r}.
\end{eqnarray*}
Since $0<1-\theta<r<p$, we can infer that $$\displaystyle\lim_{\|\varphi\| \to \infty}J_\mu(\varphi)=\infty.$$
In other words, $J_\mu$ is indeed coercive and bounded from below  on $E$. This completes the proof of Lemma \ref{coer}.
\end{proof}
\begin{lemma}\label{tphi} Assume that hypotheses $(H_1)$-$(H_3)$ hold. Then there exists a nonnegative nontrivial function $\phi \in E$ such that $J_\mu(t \phi)<0$, provided that $t>0$ is small enough.
\end{lemma}
\begin{proof}Let $t>0$ and $\phi \in C^\infty(\mathbb{R}^N)$. Assume that for some bounded subsets $\Omega_0$ and $\Omega_1$, we have  $\Omega_0\subset\mbox{supp}(\phi) \subset \Omega_1 \subset \mathbb{R}^N,$ $0\leq \phi \leq 1$ on $\Omega_1$, and $\phi =1$ on $\Omega_0$.
Then by $(H_2)$, we can  find  $K>0$, such that for all $(z,t)\in \mathbb{R}^N \times \mathbb{R}$, we have
$$f(z)H(t)\geq K f(z) |t|^r.$$
Invoking $(H_1)$-$(H_3)$ and equation \eqref{esti},  we get
\begin{eqnarray*}
J_\mu(t \phi)&=&\frac{t^p}{p}\|\phi\|^p-\frac{t^{1-\theta}}{1-\theta}\int_{\mathbb{R}^N}a(z) \phi^{1-\theta}\,dz-\mu\int_{\mathbb{R}^N}f(z) H(t \phi)\,dz\\
&\leq& \frac{t^p}{p}\|\phi\|^p-\frac{t^{1-\theta}}{1-\theta}\int_{\mathbb{R}^N}a(z) \phi^{1-\theta}\,dz-\mu K t^r \|\phi\|_{r,f}^r \\
&\leq&  t^r\left(\frac{1}{p}\|\phi\|^p+\mu K  \|\phi\|_{r,f}^r\right)-\frac{t^{1-\theta}}{1-\theta}\int_{\mathbb{R}^N}a(z) \phi^{1-\theta}\,dz\\
&\leq& t^{1-\theta}	\left[ t^{r+\theta-1}\left(\frac{1}{p}\|\phi\|^p+\mu K  \|\phi\|_{r,f}^r\right)-\frac{1}{1-\theta}\int_{\mathbb{R}^N}a(z) \phi^{1-\theta}\,dz\right]\\
&<&0,\;\;\;\mbox{ for all }\;t\in (0,\xi^{\frac{1}{r+\theta-1}}),
\end{eqnarray*}
where $$\xi=\min\left(1,\frac{\frac{t^{1-\theta}}{1-\theta}\int_{\mathbb{R}^N}a(z) \phi^{1-\theta}\,dz}{\frac{1}{p}\|\phi\|^p+\mu K  \|\phi\|_{r,f}^r}\right).$$
This completes the proof of Lemma \ref{tphi}.
\end{proof}
We note that according to  Lemma \ref{coer}, we can define the following 
$$m_\mu=\displaystyle\inf_{\varphi\in E}J_\mu(\varphi)$$
and by Lemma \ref{tphi}, we have  $m_\mu<0$.
\begin{lemma}\label{min} The functional $ J_\mu$ attains its global minimizer on $E$. That is, there exists $\varphi_\mu \in E$ such that $$J_\mu(\varphi_\mu)=m_\mu<0.$$
\end{lemma}
\begin{proof} Let $\{\varphi_{n}\}$ be a minimizing sequence for $J_\mu$, which means that $J_\mu(\varphi_{n}) \to m_\mu$, as $n\to \infty$. Since $J_\mu$ is coercive, it follows that $\{\varphi_{n}\}$ is bounded on $E$. Indeed, if not, then up to a subsequence,  we can assume that $\|\varphi_{n}\| \to \infty$. Therefore, the coercivity of $J_\mu$, implies that $J_\mu(\varphi_{n})\to \infty$, which is a contradiction. Hence $\{\varphi_{n}\}$ is bounded. Therefore, there exist $\varphi_\mu \in E$ and  a subsequence still denoted by $\{\varphi_{n}\}$, such that, as $n$ tends to infinity, we have
\begin{equation}\label{conv}
  \begin{array}{ll}
\varphi_{n} \hookrightarrow \varphi_\mu, \mbox{ weakly in }\;\;E, \\
\varphi_{n} \to \varphi_\mu, \mbox{ strongly  in }\;\;L^r(\mathbb{R}^N, f), \\
\varphi_{n} \to  \varphi_\mu, \mbox{ a.e.  in }\;\;\mathbb{R}^N.
  \end{array}
\end{equation}
Since $\{\varphi_{n}\}$ is bounded on $E$, it follows by the Sobolev embedding theorem, that $\{\varphi_{n}\}$ is
 bounded on  $L^{p^\ast}(\mathbb{R}^N)$. On the other hand, by Remark \ref{rem}, we have
$$\int_{\mathbb{R}^N}a(z)|\varphi_{n}|^{1-\theta}\,dz\leq S_{p^\ast}^{-\frac{1-\theta}{p}}|a|_{\frac{p^\ast}{p^\ast+\theta-1}}\|\varphi_{n}\|^{1-\theta}.$$
So, by  absolute continuity of $|a|_{\frac{p^\ast}{p^\ast+\theta-1}}$, we can deduce that 
$$\displaystyle\left\{\int_{\mathbb{R}^N}a(z)|\varphi_{n}|^{1-\theta}\,dz, \;n\in \mathbb{N}\right\}$$
 is equi-absolutely continuous. Therefore, by  the Vitali theorem $($see Brooks \cite{zhik}$)$, one has
\begin{equation}\label{lg}
\displaystyle\lim_{n\to \infty}\int_{\mathbb{R}^N}a(z)|\varphi_{n}|^{1-\theta}\,dz=\int_{\mathbb{R}^N}a(z)|\varphi_\mu|^{1-\theta}\,dz.
\end{equation}
Finally, by \eqref{conv} and  weak lower semi-continuity of the norm, we obtain
$$m_\mu \leq J_\mu(\varphi_\mu) \leq \displaystyle\lim_{n \to \infty}J_\mu(\varphi_{n}) =m_\mu,$$
hence \begin{equation}\label{nont}J_\mu(\varphi_\mu)=m_\mu<0.\end{equation}
This completes the proof of Lemma \ref{min}. \qed

Now we are ready to present the proof of Theorem \ref{thm}.

{\bf Proof of Theorem \ref{thm}.}
From Lemma \ref{min}, we see that $\varphi_\mu$ is a global minimizer for $J_\mu$, hence $\varphi_\mu$ satisfies
$$
0\leq J_\mu(\varphi_\mu+t \varphi)-J_\mu(\varphi_\mu), \mbox{ for all } (t,\varphi)\in (0,\infty)\times E.
$$
Dividing the above inequality  by $t>0$ and letting $t$ tend to zero, we obtain
\begin{eqnarray*}
0&\leq&\int_{\mathbb{R}^N}|\Delta \varphi_\mu|^{p-2}\Delta \varphi_\mu \Delta \varphi-\lambda \frac{|\varphi_\mu|^{p-2}\varphi_\mu \varphi}{|z|^{2p}}+|\nabla \varphi_\mu|^{p-2}\nabla \varphi_\mu \nabla \varphi \;dz\\
&&-\int_{\mathbb{R}^N}a(z) \varphi_\mu^{-\theta} \varphi \,dz-\mu \int_{\mathbb{R}^N} f(z) h(\varphi_\mu) \varphi\,dz.
\end{eqnarray*}
The fact that  $\varphi$  is arbitrary in $ E$, implies that  in the last inequality we can replace $\varphi$  by $-\varphi$, so for any $\varphi \in E$ we get
\begin{eqnarray*}
0&=&\int_{\mathbb{R}^N}|\Delta \varphi_\mu|^{p-2}\Delta \varphi_\mu \Delta \varphi-\lambda \frac{|\varphi_\mu|^{p-2}\varphi_\mu \varphi}{|z|^{2p}}+|\nabla \varphi_\mu|^{p-2}\nabla \varphi_\mu \nabla \varphi \;dz\\
&&-\int_{\mathbb{R}^N}a(z) \varphi_\mu^{-\theta} \varphi \,dz-\mu \int_{\mathbb{R}^N} f(z) h(\varphi_\mu) \varphi\,dz.
\end{eqnarray*}
That is, $\varphi_\mu$  is a weak solution for problem \eqref{p}. Moreover, from equation \eqref{nont} we see that $\varphi_\mu$ is nontrivial.
 This completes the proof of Theorem  \ref{thm}.
\end{proof}

\section{Fibering maps on Nehari manifold sets}\label{s00}
In order to prove Theorem \ref{thmm}, we first need
to study the fibering maps on  Nehari manifold sets. 
First, let us mention that the functional  $J_\mu$ defined in equation \eqref{jmu} is Fr\'echet differentiable. Moreover,
 for all $(\varphi,\psi)\in E\times E,$ we have
\begin{eqnarray*}J'_\mu(\varphi) \psi&=&\int_{\mathbb{R}^N}|\Delta \varphi|^{p-2}\Delta \varphi \Delta \psi -\lambda \frac{|\varphi|^{p-2}\varphi \psi}{|z|^{2p}}+|\nabla \varphi|^{p-2}\nabla \varphi \nabla \psi \,dz\\
&&-\int_{\mathbb{R}^N}a(z) \varphi^{-\theta} \psi\,dz-\frac{\mu}{r}\int_{\mathbb{R}^N} g(z,\varphi) \psi\,dz.
\end{eqnarray*}
It is obvious that  $J_\mu$ is not bounded from below on $E$. We introduce the following set
$$
 N_\mu=\left\{\varphi\in E; \;J_\mu'(\varphi) \varphi =0\right\}.
$$
Note that a function $\varphi\in E$ is a weak solution for problem \eqref{p}, if it   satisfies $J'_\mu(\varphi)=0$, that is, $\varphi$ is a critical value for $J_\mu$. Clearly, $\varphi\in N_\mu$ if and only if
\begin{equation}
 \label{u in Nnu}
 \|\varphi\|^p-\int_{\mathbb{R}^N}a(z) \varphi^{1-\theta}\,dz-{\mu}\int_{\mathbb{R}^N} G(z,\varphi(z))\,dz=0.
\end{equation}
\begin{lemma}
 \label{coercive}
The functional $J_\mu$ is coercive  and  bounded from below on $N_\mu$.
\end{lemma}
\begin{proof}
 Let $\varphi\in N_\mu$. Then, by equations \eqref{const}, \eqref{u in Nnu} and the  H\"{o}lder inequality, we obtain
 \begin{eqnarray}
    J_\mu(\varphi)&=&\frac{1}{p}\|\varphi\|^p-\frac{1}{1-\theta}\int_{\mathbb{R}^N}a(z) \varphi^{1-\theta}\,dz-\frac{\mu}{r}\int_{\mathbb{R}^N} G(z,\varphi(z))\,dz\nonumber\\
             &\ge& \frac{r-p}{pr}\|\varphi\|^p-\frac{\theta+r-1}{r(1-\theta)}\int_{\mathbb{R}^N}a(z) |\varphi|^{1-\theta}\,dz\nonumber\\
             &\ge& \frac{r-p}{pr}\|\varphi\|^p-\frac{\theta+r-1}{r(1-\theta)}S_p^\frac{\theta-1}{p}\|a\|_\frac{p}{\theta+p-1}\|\varphi\|^{1-\theta}\label{eq coer}
 \end{eqnarray}
Since $0<1-\theta<1<p<r$, it follows that   $J_\mu$ is coercive and bounded from below on $N_\mu$.
This completes the proof of Lemma \ref{coercive}.
\end{proof}
Next, we define a function $\phi_{\mu,\varphi}$ on $[0, +\infty),$ introduced in 
Dr\v{a}bek and Poho\v{z}aev  \cite{Drabek},   
as follows
\[
 \phi_{\mu,\varphi}(t):=J_\mu(t\varphi)=\frac{t^p}{p}\|\varphi\|^p-\frac{t^{1-\theta}}{1-\theta}\int_{\mathbb{R}^N}a(z) \varphi^{1-\theta}\,dz-\frac{\mu t^r}{r}\int_{\mathbb{R}^N} G(z,\varphi(z))\,dz.
\]
A simple calculation shows that
$$\phi'_{\mu,\varphi}(t)=t^{p-1}\|\varphi\|^p-{t^{-\theta}}\int_{\mathbb{R}^N}a(z) \varphi^{1-\theta}\,dz-{\mu t^{r-1}}\int_{\mathbb{R}^N} G(z,\varphi(z))\,dz,$$
and
$$\phi''_{\mu,\varphi}(t)=(p-1)t^{p-2}\|\varphi\|^p+{\theta t^{-\theta-1}}\int_{\mathbb{R}^N}a(z) \varphi^{1-\theta}\,dz-{\mu(r-1)t^{r-2}}\int_{\mathbb{R}^N} G(z,\varphi(z))\,dz.$$
Since $t\phi'_{\mu,\varphi}(t)=<J_\mu'(t \varphi),t \varphi>$, it follows that for $t>0$ and $\varphi\in E\setminus\{0\}$, we have
$$\phi'_{\mu,\varphi}(t)=0 \mbox{ if and only if } t \varphi\in N_\mu.$$
In particular, $\varphi\in N_\mu$ if and only if $\phi'_{\mu,\varphi}(1)=0$. On the other hand, it follows
by equation \eqref{u in Nnu},  that for all $\varphi\in N_{\mu},$ one has
\begin{eqnarray}
 \phi''_{\mu,\varphi}(1)  &=& (p-r)\|\varphi\|^p+{(\theta+r-1) }\int_{\mathbb{R}^N}a(z) \varphi^{1-\theta}\,dz\label{eq4.4} \\
                 &=& (\theta+p-1)\|\varphi\|^p-\mu{(\theta+r-1) }\int_{\mathbb{R}^N}G(z,\varphi(z))\,dz.\label{eq4.5}
\end{eqnarray}
\\
Now, in order to obtain the multiplicity of solutions, we split $N_\mu$ into three  parts
\[
 N_\mu^+=\left\{\varphi\in N_\mu\backslash\{0\}; \phi_{\mu,\varphi}''(1)>0\right\},
\]
\[
 N_\mu^-=\left\{\varphi\in N_\mu\backslash\{0\}; \phi_{\mu,\varphi}''(1)<0\right\},
\]
and
\[
 N_\mu^0=\left\{\varphi\in N_\mu\backslash\{0\}; \phi_{\mu,\varphi}''(1)=0\right\}.
\]
In the following lemmas we shall present some important properties related to the   subsets
introduced above.
\begin{lemma}
 \label{Lagrange}
 If 
 $u\not\in N^0_\mu$ is a local mimimizer for $J_\mu$ on $N_\mu$, then $J'_\mu(\varphi)=0$.
\end{lemma}
\begin{proof}
  Since $\varphi$ is a minimizer for $J_\mu$ under the following  constraint
  $$I_\mu(\varphi):= J'_\mu(\varphi) \varphi=0,$$
 the  Lagrange multipliers theory implies the existence of  $\xi\in\mathbb{R}$ such that $J'_\mu(\varphi)= I'_\mu(\varphi) \xi$. Thus
$$J'_\mu(\varphi) \varphi =(I'_\mu(\varphi) \varphi) \xi=\phi''_{\mu,\varphi}(1) \xi=0.$$
The fact that  $\varphi\not\in N^0_\mu$, implies that $\phi''_{\mu, \varphi}(1)\neq0$.  So, $\xi=0$, which completes the proof of Lemma \ref{Lagrange}.
\end{proof}

\begin{lemma}
 \label{mu0}
 There exists $\mu_0$ such that if $\mu\in(0,\mu_0)$ then  the set $N^0_\mu$ is empty.
\end{lemma}
\begin{proof}Put
$$\mu_0=\frac{(\theta+p-1)S_r^\frac{r}{p}}{(\theta+r-1)M}\left(\frac{r-p}{(\theta+r-1)\|a\|_{\frac{p}{\theta+p-1}}S_p^\frac{1-\theta}{p}}\right)^\frac{r-p}{\theta+p-1},
$$
where $M$ is defined as in equation \eqref{M}, and let $\mu\in(0,\mu_0)$. We shall prove that $N^0_\mu=\varnothing$. Suppose to the contrary and let  $\varphi\in N^0_\mu$. Then we have
  $$0=\phi''_{\mu,\varphi}(1)=(p-1)\|\varphi\|^p+{\theta}\int_{\mathbb{R}^N}a(z) \varphi^{1-\theta}(z)\,dz-{\mu(r-1)}\int_{\mathbb{R}^N} G(z,\varphi(z))\,dz.$$
  So, it follows from \eqref{eq4.4} and \eqref{eq4.5} that
  \begin{equation}
  \label{eq4.6}
  (\theta+p-1)\|\varphi\|^p=\mu{(\theta+r-1) }\int_{\mathbb{R}^N}G(z,\varphi(z))\,dz,
  \end{equation}
and
  \begin{equation}
  \label{4.7}
   (r-p)\|\varphi\|^p={(\theta+r-1) }\int_{\mathbb{R}^N}a(z) \varphi^{1-\theta}(z)\,dz.
  \end{equation}
On the other hand, from \eqref{const} and the  H\"{o}lder inequality, we get
\begin{eqnarray*}
      \int_{\mathbb{R}^N}a(z) \varphi^{1-\theta}(z)\,dz&\le&\left(\int_{\mathbb{R}^N}|\varphi(z)|^p\,dz\right)^\frac{1-\theta}{p}\left(\int_{\mathbb{R}^N}|a(z)|^\frac{p}{\theta+p-1}\,dz\right)^\frac{\theta+p-1}{p} \\
                                              &\le&|\varphi|_p^{1-\theta}\|a\|_\frac{p}{\theta+p-1} 
                                              \le S_p^\frac{1-\theta}{p}\|a\|_\frac{p}{\theta+p-1}\|\varphi\|^{1-\theta}.
\end{eqnarray*}
So, it follows from \eqref{4.7} that
\begin{eqnarray*}
 \|\varphi\|^p&=&\frac{\theta+r-1}{r-p}\int_{\mathbb{R}^N}a(z) u^{1-\theta}(z)\,dz
        \le \frac{\theta+r-1}{r-p}S_p^\frac{1-\theta}{p}\|a\|_\frac{p}{\theta+p-1}\|\varphi\|^{1-\theta},
\end{eqnarray*}
that is,
\begin{equation}
 \label{*}
 \|\varphi\|\le\left(\frac{\theta+r-1}{r-p}S_p^\frac{\theta-1}{p}\|a\|_\frac{p}{\theta+p-1}\right)^\frac{1}{\theta+p-1}.
\end{equation}
From \eqref{const}, \eqref{M}, and \eqref{eq4.6}, we have
 \begin{eqnarray*}
 \|\varphi\|^p&=&\mu\frac{{(\theta+r-1) }}{\theta+p-1}\int_{\mathbb{R}^N}G(z,\varphi(z))\,dz\\
        &\le&\mu M\frac{{(\theta+r-1) }}{\theta+p-1}\int_{\mathbb{R}^N}|\varphi(z)|^r\,dz
        \le \mu M\frac{{(\theta+r-1) }}{\theta+p-1}S_r^{-\frac{r}{p}}\|\varphi\|^r,
\end{eqnarray*}
hence,
\begin{equation}
 \label{**}
 \|\varphi\|\ge\left(\frac{(\theta+p-1)S_p^\frac{r}{p}}{(\theta+r-1)M\mu}\right)^\frac{1}{r-p}.
\end{equation}
By combining \eqref{*} with  \eqref{**}, we obtain $\mu\ge\mu_0$,  which gives us the desired contradiction. This completes the proof of Lemma \ref{mu0}.
\end{proof}
 \begin{lemma}\label{intG>0}
Let  $\varphi\in E\setminus\{0\}$. Then there exists $\mu_1>0$ such that for all
\
$0<\mu<\mu_1,$   \
$ \phi_\varphi$ has exactly a local minimum at  $t_1$ and a local maximum at $t_2$. That is, $t_1u\in N^+_\mu$ and $t_2u\in N^-_\mu$.
\end{lemma}

\begin{proof}
Let  $\varphi\in E$ be such that   
$$\int_{\mathbb{R}^N}g(z,\varphi)dz>0
\hbox{ and }
\int_{\mathbb{R}^N}a(z)\varphi^{1-\theta}dz>0.$$
 It is easy to see that for all $t>0$, we have
 \begin{equation}
 \label{phim}
 \phi'_{\mu, \varphi}(t)=t^{-\theta}\left(m_\varphi(t)-\int_{\mathbb{R}^N}a(z)\varphi^{1-\theta}dz\right),
\end{equation}
where  $m_\varphi:[0,\infty)\to \mathbb R$ is defined by
\begin{equation*}
 \label{m(t)}
   m_\varphi(t)=t^{\theta+p-1}\|\varphi\|^p-t^{\theta+r-1}  \int_{\mathbb{R}^N}g(z,\varphi)dz.
 \end{equation*}
It is not difficult to show that   $m'_\varphi(t)=0$ if and only if $t=0$ or $t=t_0$, where
\begin{equation}
 \label{t0}
t_0=\left(\frac{(\theta+p-1)\|\varphi\|^p}{(\theta+r-1)\mu\int_{\mathbb{R}^N}g(z,\varphi)dz}\right)^\frac{1}{r-p}.
\end{equation}
 Moreover,
 \begin{equation}
  \label{m(t0)}
  m_\varphi(t_0)=\left(\mu\int_{\mathbb{R}^N}g(z,\varphi)dz\right)^{-\frac{\theta+p-1}{r-p}}\left(\left(\frac{\theta+p-1}{\theta+r-1}\right)^\frac{\theta+p-1}{r-p}-\left(\frac{\theta+p-1}{\theta+r-1}\right)^\frac{\theta+r-1}{r-p}\right)>0.
 \end{equation}
On the other hand, the table of variation of the function $ m_\varphi$ is given by
$$\small
{ \begin{array}{|c|ccccr|}
\hline
t     &0 &  & t_0 & &  \infty \\ \hline
 m_\varphi'(t) &  & + & 0  & - &        \\ \hline
      &   &   & m_\varphi(t_0) &   &           \\ 
 m_\varphi(t) &    &\nearrow & &\searrow  & \\ 
     &0 &  &  & & -\infty       \\ 
\hline
\end{array}}
$$
Now, since  $$\displaystyle0<\int_{\mathbb{R}^N}a(z) \varphi^{1-\theta}\,dz\le\frac{\theta+r-1}{r-p}S_p^\frac{\theta-1}{p}\|a\|_\frac{p}{\theta+p-1}\|\varphi\|^{1-\theta},$$
it follows by \eqref{m(t0)} that we can choose  $\mu_1>0$ small enough, so that  for all $\mu\in(0,\mu_1)$ we have
$$\frac{\theta+r-1}{r-p}S_p^\frac{\theta-1}{p}\|a\|_\frac{p}{\theta+p-1}\|\varphi\|^{1-\theta}< m_\varphi(t_0).$$
Therefore  for $\mu\in(0,\mu_1)$ we have,
$$\displaystyle0<\int_{\mathbb{R}^N}a(z) \varphi^{1-\theta}\,dz<m_\varphi(t_0).$$
Hence, from the table of variation of $ m_\varphi$, we can deduce the existence of unique $t_1$ and $t_2$ such that $0<t_1<t_0<t_2$ and
$$m_\varphi(t_1)=m_\varphi(t_2)=\int_{\mathbb{R}^N}a(z) \varphi^{1-\theta}\,dz.$$
Finally, from \eqref{phim} and the table of variation of  function $ m_\varphi$, we can see that $t_1$ and $t_2$ are the unique critical points of  function $ \phi_{\mu,u}$. More precisely, $t_1$  is a local minimum point and $t_2$ is a local maximum point.  Thus $t_1u\in N^+_\mu$ and $t_2u\in N^-_\mu$. This completes the proof of Lemma \ref{intG>0}.
\end{proof}
\begin{remark}
 It follows from Lemma \ref{intG>0} that $N^+_\mu\neq \varnothing$ and $N^-_\mu\neq \varnothing$,  provided that $0<\mu<\mu_1$.  Moreover, by Lemma \ref{mu0}, for every $0<\mu<\mu_0$, we have
 $$N_\mu=N^+_\mu\cup N^-_\mu.$$
\end{remark}
For the rest of the paper we shall set $$\mu^*=\min(\mu_0, \mu_1, \mu_2),$$  and define
$$\theta_\mu=\inf_{\varphi\in N_\mu}J_\mu(\varphi), \theta^+_\mu=\inf_{\varphi\in N^+_\mu}J_\mu(\varphi) \mbox{ and } \theta^-_\mu=\inf_{\varphi\in N^-_\mu}J_\mu(\varphi),$$
where
$$\mu_2=\frac{(\theta+p-1)S_r^\frac{r}{p}}{(\theta+r-1)M}\left(\frac{(\theta+r-1)p}{(1-\theta)(r-p)}S_p^\frac{\theta-1}{p}\|a\|_\frac{p}{\theta+p-1}\right)^\frac{r-p}{\theta+p-1}.$$
\begin{lemma}
 \label{theta+-}
 If
 \
  $0<\mu<\mu^*$, then the following statments hold
 \begin{enumerate}
  \item[(i)]  
             $$ \theta_\mu\le\theta^+_\mu<0.$$
              
  \item[(ii)] There exists $C>0$ such that  
                                             $$ \theta^-_\mu\ge C>0.$$                                               
 \end{enumerate}
\end{lemma}
\begin{proof}
 \begin{enumerate}
  \item[(i)] Let $\varphi\in N^+_\mu$. Then, from \eqref{eq4.4}, we get
            \begin{equation*}
             \frac{r-p}{\theta+r-1}\|\varphi\|^p<\int_{\mathbb{R}^N}a(z)\varphi^{1-\theta}dz.
            \end{equation*}
            So combining the last inequality with \eqref{u in Nnu}, we obtain
            \begin{eqnarray*}
             J_\mu(\varphi)&=&\frac{r-p}{p r}\|\varphi\|^p-\frac{\theta+r-1}{r(1-\theta)}\int_{\mathbb{R}^N}a(z)\varphi^{1-\theta}dz\\
                  &\le&-\frac{(r-p)(\theta+p-1)}{p r(1-\theta)}\|\varphi\|^p<0,
            \end{eqnarray*}
            so we conclude that $\theta_\mu\le\theta^+_\mu<0$.
  \item[(ii)] Let $\varphi\in N^-_\mu$. Then by \eqref{const} and \eqref{eq4.5} we get
              $$\|\varphi\|>\left( \frac{(\theta+p-1)S_r^\frac{r}{p}}{(\theta+r-1)\mu M}\right)^\frac{1}{r-p},$$
              where $M$ is the positive constant given by equation \eqref{M}. \\
Now, using the last inequality and   \eqref{eq coer} we get
            \begin{eqnarray*}
        J_\mu(\varphi)&\ge&\frac{r-p}{p r}\|\varphi\|^p-\frac{\theta+r-1}{r(1-\theta)}S_p^\frac{1-\theta}{p}\|a\|_\frac{p}{\theta+p-1}\|\varphi\|^{1-\theta}\\
                &\ge&\|\varphi\|^{1-\theta}\left(\frac{r-p}{p r}\|\varphi\|^{\theta+p-1}-\frac{\theta+r-1}{r(1-\theta)}S_p^\frac{1-\theta}{p}\|a\|_\frac{p}{\theta+p-1}\right)\\
                &>&\left( \frac{(\theta+p-1)S_r^\frac{r}{p}}{(\theta+r-1)\mu M}\right)^\frac{1-\theta}{r-p}\\
&&\left(\frac{r-p}{p r}\left( \frac{(\theta+p-1)S_r^\frac{r}{p}}{(\theta+r-1)\mu M}\right)^\frac{\theta+p-1}{r-p}-\frac{\theta+r-1}{r(1-\theta)}S_p^\frac{\theta-1}{p}\|a\|_\frac{p}{\theta+p-1}\right).
            \end{eqnarray*}
Since $0<\mu<\mu^*\leq \mu_2$ and $0<1-\theta\leq p<r$, it follows that $J_\mu>C,$ for some $C>0$.
 \end{enumerate}
 This completes the proof of Lemma \ref{theta+-}.
\end{proof}
Next,  we have the following results on the existence of minimizers in $N^+_\mu$ and $N^-_\mu$ for $\mu\in(0,\mu^*)$.
\begin{lemma}
 \label{minimu+}
  If
  \
  $0<\mu<\mu^*$, then there exists $\varphi_\mu\in N^+_\mu$ such that
$$\theta^+_\mu=J_\mu(\varphi_\mu).$$ That is, $J_\mu$ attains its minimum on $N^+_\mu$.

\end{lemma}
\begin{proof}
 Since $J_\mu$ is bounded from below on $N_\mu$ and hence also on $N^+_\mu$,  there exists  $\{\varphi_k\}\subset N^+_\mu$ such that
$$\lim_{k\to\infty}J_\mu(\varphi_k)=\inf_{\varphi\in N^+_\mu}J_\mu(\varphi).$$
Since $J_\mu$ is coercive on $N_\mu$, it follows that  $\{\varphi_k\}$ is bounded on $E$. So,  there exist $\varphi_\mu$ and a subsequence, again  denoted  by $\{\varphi_k\},$ such that as $k$ tends to infinity, we have
$$\begin{cases}
 \varphi_k \rightharpoonup \varphi_\mu \mbox{ weakly in } E \\
 \varphi_k\longrightarrow \varphi_\mu \mbox{ strongly in } L^q(\mathbb{R}^N), \mbox{for all}\;\;p<q<p^*,\\
 \varphi_k\longrightarrow \varphi_\mu \mbox{ a.e } \mathbb{R}^N.
\end{cases}$$

From Lemma \ref{theta+-} we know that  $\displaystyle\inf_{u\in N^+_\mu}J_\mu(\varphi)<0$. On the other hand, since $\{\varphi_k\}\subset N_\mu$ we have
$$J_\mu(\varphi_k)=\frac{r-p}{p r}\|\varphi_k\|^p-\frac{\theta+r-1}{r(1-\theta)}\int_{\mathbb{R}^N}a(z)\varphi_k^{1-\theta}(z)dz,$$
so we get
$$\displaystyle\frac{\theta+r-1}{r(1-\theta)}\int_{\mathbb{R}^N}a(z)\varphi_k^{1-\theta}(z)dz=\frac{r-p}{p r}\|\varphi_k\|^p-J_\mu(\varphi_k).$$
From  \eqref{lg},  by letting  $k\to\infty$ in the last equation, we obtain
$$
 \int_{\mathbb{R}^N}a(z)\varphi_\mu^{1-\theta}(z)dz>0.
$$

We now  claim that $\varphi_k$ converges strongly to $ \varphi_\mu$ in $E$. If this were not true, then we would have
$$\|\varphi_\mu\|^p<\liminf_{k\to\infty}\|\varphi_k\|^p.$$
Since $\phi'_{\varphi_\mu}(t_1)=0,$ it would follow
that
 $\phi'_{\varphi_k}(t_1)>0$ for sufficiently large $k$.  So, we must have $t_1>1$.
 However, $t_1\varphi_\mu\in N^+_\mu$ and therefore
$$J_\mu(t_1\varphi_\mu)<J_\mu(\varphi_\mu)\le\lim_{k\to\infty}J_\mu(\varphi_k)=\inf_{u\in N^+_\mu}J_\mu(\varphi),$$
which is a contradiction, that is $\varphi_k\underset{k\to\infty}\longrightarrow \varphi_\mu$.

Since $N^0_\mu=\varnothing$, it follows that $\varphi_\mu\in N^+_\mu$.  Finally, $\varphi_\mu$ is a minimizer for $J_\mu$ on $N^+_\mu$. This completes the proof of Lemma \ref{minimu+}.
\end{proof}
\begin{lemma}
 \label{minimu-}
  If 
  \
    $0<\mu<\mu^*$, then there exists $\psi_\mu\in N^-_\mu$ such that
$$\theta^-_\mu=J_\mu(\psi_\mu).$$ That is, $J_\mu$ achieves its minimum on $N^-_\mu$.
\end{lemma}
\begin{proof}
By Lemma \ref{theta+-},  there exists $C>0$ such that for all  $\varphi\in N^-_\mu$, we have $J_\mu(\varphi)>C$.  So, there exists a minimizing sequence
  $\{\varphi_k\}\subset N^-_\mu$ such that
  $$
   \lim_{k\to\infty}J_\mu(\varphi_k)=\inf_{\varphi\in N^-_\mu}J_\mu(\varphi)>0.
 $$
Since  $J_\mu$ is coercive, we can deduce that  $\{\varphi_k\}$ is  bounded. So, for all $p\leq r<p^*$, there is a subsequence still denoted by $\{\varphi_k\}$, and $\psi_\mu\in E$ such that if $k$ tends to infinity we get
$$\begin{cases}
 \varphi_k\rightharpoonup \psi_\mu \mbox{ weakly in } E \\
 \varphi_k\longrightarrow \psi_\mu \mbox{ strongly in } L^r(\mathbb{R}^N)\\
 \varphi_k\longrightarrow \psi_\mu \mbox{ a.e. } \mathbb{R}^N.
\end{cases}$$

On the other hand, since $\{\varphi_k\}\subset N_\mu$ we have
$$J_\mu(\varphi_k)=\mu\frac{r+\theta-1}{r(1-\theta)}\int_{\mathbb{R}^N}G(z,\varphi_k(z))dz- \frac{\theta+p-1}{p(1-\theta)} \|\varphi_k\|^p,$$
which implies
$$\mu\frac{r+\theta-1}{r(1-\theta)}\int_{\mathbb{R}^N}G(z,\varphi_k)dz=J_\mu(\varphi_k)+ \frac{\theta+p-1}{p(1-\theta)}\|\varphi_k\|^p.$$

By letting  $k\to\infty$ in last equation, we obtain
$$
 \int_{\mathbb{R}^N}G(z,\psi_\mu)dz>0.
$$
Hence, by Lemma \ref{intG>0}  $\phi_{\mu, \varphi}$  has a maximum at some point $t_2$ and $t_2\psi_\mu\in N^-_\mu$. On the other hand,  $\psi_k\in N^-_\mu$ implies that $1$ is a global maximum point for $\phi_{\mu, \varphi_k}$, so we get
\begin{equation}
 \label{Eq3.5}
 J_\mu(t\varphi_k)=\phi_{\mu, \varphi_k}(t)\le\phi_{\mu, \varphi_k}(1)= J_\mu(\varphi_k), \;\mbox{for all }\;t>0.
\end{equation}

Now, we claim that  $\varphi_k\underset{k\to\infty}\longrightarrow \psi_\mu$. Suppose that  this is were not true, then we would get
$$\|\psi_\mu\|^p<\liminf_{k\to\infty}\|\varphi_k\|^p.$$
So, from equation \eqref{Eq3.5} and the Fatou lemma  we would obtain
\begin{eqnarray*}
 J_\mu(t_2\psi_\mu)&=&\frac{t_2^p}{p}\|\psi_\mu\|^p-\frac{t_2^{1-\theta}}{1-\theta}\int_{\mathbb{R}^N}a(z) \psi_\mu^{1-\theta}\,dz-\frac{\mu t_2^r}{r}
                   \int_{\mathbb{R}^N} G(z,\psi_\mu(z))\,dz\\
                &<&\liminf_{k\to\infty}\left(\frac{t_2^p}{p}\|\varphi_k\|^p-\frac{t_2^{1-\theta}}{1-\theta}\int_{\mathbb{R}^N}a(z) \varphi_k^{1-\theta}\,dz-\frac{\mu t_2^r}{r}\int_{\mathbb{R}^N} G(z,\varphi_k(z))\,dz\right)\\
                &\le&\lim_{k\to\infty}J_\mu(t_2\varphi_k)\\
                &\le&\lim_{k\to\infty}J_\mu(\varphi_k)=\inf_{\varphi\in N^-_\mu}J_\mu(\varphi),
\end{eqnarray*}
which is a contradiction. Hence, $\varphi_k\longrightarrow \psi_\mu$ as $k\to\infty$.

Since $N^0_\mu=\varnothing$, it follows that $\psi_\mu\in N^-_\mu$.  Finally, $\psi_\mu$ is a minimizer for $J_\mu$ on $N^-_\mu$. This completes the proof of Lemma \ref{minimu-}.
\end{proof}

\section{The proof of Theorem \ref{thmm}}\label{s3}

We shall  need the following two auxiliary lemmas to prove that  the  local minimum of the functional energy is a weak solution for problem \eqref{p}.
\begin{lemma}\label{impl1}Assume that hypotheses of Theorem \ref{thmm} are satisfied and $\mu\in(0,\mu^\star)$. Then the following statments hold:
\begin{enumerate}
\item[(i)]There exist $r_1>0$ and a continuous function $\rho_1: B(0,r_1)\to (0,\infty)$ such that
$$\rho_1(0)=1\;\;\mbox{and}\;\;\rho_1(\varphi)(\varphi_\mu+\varphi)\in N^+_\mu,\;\mbox{for all}\;\varphi\in B(0,r_1).$$
\item[(ii)]There exist $r_2>0$ and a continuous function $\rho_2: B(0,r_2)\to (0,\infty)$ such that
$$\rho_2(0)=1\;\;\mbox{and}\;\;\rho_2(\varphi)(\psi_\mu+\varphi)\in N^-_\mu,\;\mbox{for all}\;\varphi\in B(0,r_2).$$
\end{enumerate}
\end{lemma}
\begin{proof}We give the proof only for assertion $(i)$, since the proof for assertion $(ii)$ is similar. So, let $\Phi: E\times (0,\infty)$ be a function defined by
$$\Phi(\varphi,t)=t^{\theta+p-1}\|\varphi_\mu+\varphi\|^p-t^{\theta+r-1}  \int_{\mathbb{R}^N}G(z, \varphi_\mu+\varphi)dz-\int_{\mathbb{R}^N}a(z)|\varphi_\mu+\varphi|^{1-\theta}dz. $$
Since $\varphi_\mu\in N^+_\mu\subset N_\mu$,  we have $\Phi(0,1)=0$.  On the other hand,  $\varphi_\mu\in N^+_\mu$ implies that
$$\frac{\partial \Phi}{\partial t}(0,1)=(\theta+p-1)\|\varphi_\mu\|^p-(\theta+r-1)  \int_{\mathbb{R}^N}G(z, \varphi_\mu)dz>0.$$
So by the Implicit function theorem, there exist $r_1>0$ and a continuous function $\rho_1: B(0,r_1)\to (0,\infty)$ such that
$$\rho_1(0)=1\;\;\mbox{and}\;\;\rho_1(\varphi)(\varphi_\mu+\varphi)\in N^+_\mu,\;\mbox{for all}\;\varphi\in B(0,r_1).$$
This completes the proof of Lemma \ref{impl1}.
\end{proof}

\begin{lemma}\label{impl2}Assume that hypotheses of Theorem \ref{thmm} are satisfied and $\mu\in(0,\mu^\star)$. Then for every  $\varphi\in E$  the following statments hold:
\begin{enumerate}
\item[(i)]There exists  $T_1>0$ such that
$$J_\mu(\varphi_\mu)\leq J_\mu(\varphi_\mu+t \varphi), \;\mbox{for all}\;t\in(0,T_1).$$
\item[(ii)]There exists $T_2>0$ such that
$$J_\mu(\psi_\mu)\leq J_\mu(\psi_\mu+t \varphi), \;\mbox{for all}\;t\in(0,T_2).$$
\end{enumerate}
\end{lemma}
\begin{proof}We shall give the proof only for assertion $(i)$, since the proof for assertion $(ii)$ is similar. So, let  $\varphi\in E$ and $\delta_\varphi:[0,\infty)\to \mathbb{R}$ be a function defined by
$$\delta_\varphi(t)=(p-1)\|\varphi_\mu+t\varphi\|^p+\theta \int_{\mathbb{R}^N}a(z)|\varphi_\mu+t\varphi|^{1-\theta}dz-(r-1)  \int_{\mathbb{R}^N}G(z, \varphi_\mu+t\varphi)dz.$$
Since  $\varphi_\mu\in N^+_\mu\subset N_\mu $, we obtain
\begin{equation}\label{theta1}
\theta \int_{\mathbb{R}^N}a(z)|\varphi_\mu|^{1-\theta}dz=\theta \|\varphi_\mu\|^p+(r-1)  \int_{\mathbb{R}^N}G(z, \varphi_\mu)dz,
\end{equation}
and
\begin{equation}\label{theta2}
(\theta+p-1)\|\varphi_\mu\|^p-(\theta+r-1)  \int_{\mathbb{R}^N}G(z, \varphi_\mu+t\varphi)dz>0.
\end{equation}

By combining equations \eqref{theta1} and \eqref{theta2} with the definition of the function $\delta_\varphi$, we get  $\delta_\varphi(0)>0.$  So the continuity of the function $\delta_\varphi$ implies the existence of $T_0>0$ such that
 $$\delta_\varphi(t)>0, \mbox{ for all } t\in[0,T_0].$$
On the other hand, by Lemma \ref{impl1}, for every $t\in[0,r_1]$, there exists $\overline{\rho_1}(t)$ such that
\begin{equation}\label{01}
\overline{\rho_1}(t)(\varphi_\mu+t\varphi)\in N^+_\mu \mbox{ and } \displaystyle\lim_{t\to 0^+}\overline{\rho_1}(t)=1.
\end{equation}
Moreover, by Lemma \ref{minimu+}, we have
$$\theta_\mu^+=J_\mu(\varphi_\mu)\leq J_\mu(\overline{\rho_1}(t)(\varphi_\mu+t\varphi)), \;\mbox{for all}\;t\in(0,T_0).$$

Now, from that fact that $\Phi_{\mu, \varphi_\mu}''(1)>0$ and the continuity in $t$ we get 
$$\Phi_{\mu, \varphi_\mu+t\varphi}''(1)>0,
\hbox{ for all }
t\in [0,T_1]
\hbox{ and for some small enough }
T_1\in(0,T_0).$$
 So using equation \eqref{01}, we can get  small enough $T_1\in(0,T_0)$  such that
$$\theta_\mu^+=J_\mu(\varphi_\mu)\leq J_\mu (\varphi_\mu+t\varphi)), \;\mbox{for all}\;t\in[0,T_1).$$
This completes the proof of Lemma \ref{impl2}.
\end{proof}

Now we are ready to present the proof of Theorem \ref{thmm}.\\

{\bf Proof of Theorem \ref{thmm}.}\\
As a direct consequence of  Lemma \ref{minimu+} and Lemma \ref{minimu-}, we can deduce that $J_\mu$ has  minimizers   $\varphi_\mu\in N^+_\mu$  and  $\psi_\mu\in N^-_\mu$. Moreover, $N^+_\mu\cap N^-_\mu=\varnothing$ implies that  $\varphi_\mu$  and $\psi_\mu$ are distinct.

Next, we shall prove that $\varphi_\mu$  and $\psi_\mu$ are weak solutions for problem \eqref{p}. To this end, let  $\varphi\in E$. Then by the assertion $(i)$ of Lemmas \ref{impl1}, \ref{impl2}, we obtain
$$0\leq J_\mu(\varphi_\mu+t \varphi)-J_\mu(\varphi_\mu), \;\mbox{for all}\;t\in(0,T_1).$$
Dividing the last inequality by $t$ and letting $t$ tend to zero, we get
$$\int_{\mathbb{R}^N}|\Delta \varphi_\mu|^{p-2}\Delta \varphi_\mu \Delta \varphi-\lambda \frac{|\varphi_\mu|^{p-2}\varphi_\mu \varphi}{|z|^{2p}}+|\nabla \varphi_\mu|^{p-2}\nabla \varphi_\mu \nabla \varphi \;dz$$ $$-\int_{\mathbb{R}^N}a(z) \varphi_\mu^{-\theta} \varphi \,dz-\mu \int_{\mathbb{R}^N} f(z) h(\varphi_\mu) \varphi\,dz\geq 0.$$
Since  $\varphi$  is arbitrary in $ E$, it follows that  in the last inequality we can replace $\varphi$  by $-\varphi$. So for all $\varphi \in E$ we get
\begin{eqnarray*}
0&=& \int_{\mathbb{R}^N}|\Delta \varphi_\mu|^{p-2}\Delta \varphi_\mu \Delta \varphi-\lambda \frac{|\varphi_\mu|^{p-2}\varphi_\mu \varphi}{|z|^{2p}}+|\nabla \varphi_\mu|^{p-2}\nabla \varphi_\mu \nabla \varphi \;dz\\
&-& \int_{\mathbb{R}^N}a(z) \varphi_\mu^{-\theta} \varphi \,dz-\mu \int_{\mathbb{R}^N} f(z) h(\varphi_\mu) \varphi\,dz.
\end{eqnarray*}
That is, $\varphi_\mu$  is a weak solution of problem \eqref{p}. Moreover, from equation \eqref{nont} we see that $\varphi_\mu$ is nontrivial.

Finally, if we proceed as above using assertion $(ii)$ of Lemmas \ref{impl1}
and
 \ref{impl2},  we can prove that $\psi_\mu$  is also a nontrivial  weak solution of problem \eqref{p}. This completes the proof of Theorem \ref{thmm}.\qed

\section{An application}\label{s5}
As an application of our main results, we shall consider the following problem
\begin{equation}\label{qq}
\Delta_{p}^{2}\varphi-\lambda \frac{|\varphi|^{p-2}\varphi}{|z|^{2p}}+\Delta_p \varphi=  \frac{ a(z)}{\varphi^{\theta}}+ \mu f(z)|\varphi|^{r-2}\varphi \quad \mbox{ in }\mathbb{R}^N,
\end{equation}
where $\mu>0$, $1<p<\frac{N}{2}$, $0<\theta<1,$  and $\lambda$ satisfies equation \eqref{lamb}. \\
We note that problems of type \eqref{qq}  describe e.g.,  the deformations of an elastic beam. Also, they
 give a model for studying traveling waves in suspension bridges.

First, let us assume that $1<r<p$, $f$ is a positive function in
$$ L^{\frac{p^\ast}{p^\ast-r}}(\mathbb{R}^N)\cap L^s_{loc}(\mathbb{R}^N),
\
\hbox{ for some}
\
s \in (\frac{p^\ast}{p^\ast-r}, \frac{p}{p^\ast-r}),$$
which implies that the first part of hypothesis $(H_1)$ is satisfied.

On the other hand, it is easy that the function $h(z)=|\varphi|^{r-2}\varphi$ satisfies the second part of hypothesis $(H_1)$. Moreover, a simple calculation shows that
$$0<rf(z)H(\varphi)=f(z)h(\varphi)\varphi,$$ so hypothesis $(H_2)$ is also satisfied.

 Finally, if
 $$a \in L^{\frac{p^\ast}{p^\ast+\theta-1}}(\mathbb{R}^N)\cap L^\beta_{loc}(\mathbb{R}^N),
 \
 \hbox{ for some}
 \
 \beta \in (\frac{p^\ast}{p^\ast+\theta-1}, \frac{p}{\theta+p-1}),$$
  then Theorem \ref{thm} ensures the existence of nontrivial solution for problem \eqref{qq}.

Next, we assume that $p<r<p^\ast$ and $a$ is a positive function  in $L^{\frac{p}{\theta+p-1}}(\mathbb{R}^N),$ that is, hypothesis $(H_5)$ is satisfied. It is not difficult to see that if
$$g(z,\varphi)=f(z)|\varphi|^{r-2}\varphi,$$
then
$$G(z,\varphi)=f(z)|\varphi|^{r},$$
so hypothesis $(H_4)$ is also satisfied. Hence, Theorem \ref{thmm}
now ensures the existence of two nontrivial solutions for problem \eqref{qq}.

\subsection*{Conflicts of interest.} The authors declare  no conflicts of interest.

\subsection*{Acknowledgment.} We thank the referees for their comments and suggestions.

\end{document}